\documentclass[12pt]{article}

\RequirePackage{amssymb,amsmath,amsthm}
\RequirePackage[top=2.5cm, bottom=2.5cm, left=2.5cm, right=2.5cm]{geometry}
\RequirePackage{graphicx}
\RequirePackage{fancyvrb}
\RequirePackage[normalem]{ulem} 
\RequirePackage{faktor}
\RequirePackage{nicefrac}
\RequirePackage{enumerate}
\RequirePackage{blindtext}
\RequirePackage{hyperref}
\RequirePackage{mathtools}
\RequirePackage{cleveref}
\RequirePackage{tikz}
\RequirePackage{verbatim}
\usetikzlibrary{matrix}
\usepackage{faktor}
\usepackage{enumerate}
\usepackage{dsfont}
\usepackage{marginnote}
\usepackage{multicol}

\theoremstyle{definition}
\newtheorem{definition}{Definition}[section]
\newtheorem{theorem}{Theorem}[section]
\newtheorem{proposition}{Proposition}[section]

\newtheorem{corollaryy}{Corollary}[proposition]
\newtheorem{lemma}[theorem]{Lemma}

\theoremstyle{definition}

\theoremstyle{remark}
\newtheorem*{remark}{Remark}

\newcommand{\ra}{\ensuremath{\rightarrow}}
\newcommand{\pml}{\ensuremath{\begin{pmatrix}}}
\newcommand{\pmr}{\ensuremath{\end{pmatrix}}}
\newcommand{\defeq}{\vcentcolon=}
\newcommand{\suchthat}{\text{ s.t. }}

\newcommand{\fg}{\mathfrak g}

\newcommand{\fm}{\mathfrak m}

\newcommand{\fh}{\mathfrak h}

\newcommand{\ZZ}{\mathbb{Z}}
\newcommand{\CC}{\mathbb{C}}
\newcommand{\QQ}{\mathbb{Q}}

\newcommand{\supp}{\operatorname{supp}}

\newcommand{\spec}{\operatorname{spec}}

\newcommand{\aut}{\operatorname{Aut}}
\newcommand{\Aut}{\operatorname{Aut}}

\newcommand{\ad}{\text{ad}}
\newcommand{\Ad}{\text{Ad}}

\begin{document}

\begin{titlepage}
    \begin{center}
        \vspace*{1cm}
 
        \Huge
        On a family of non-weight modules over Virasoro-type Lie (super)algebras
 
        % \vspace{0.5cm}ד
        % \LARGE
        % Thesis Subtitle
 
        \vspace{1cm}
    
        \LARGE
        Sarah Williamson
        
        \vspace{1cm}
        
        \Large
        A Thesis Submitted for the Degree of Master of Science to the Scientific Council of the Weizmann Institute of Science.
 
        \vfill
        
        \Large
        Department of Mathematics and Computer Science\\
        The Weizmann Institute of Science\\
        Israel
 
    \end{center}
\end{titlepage}

\begin{abstract}
    In this thesis we classify modules over a Witt-type Lie algebra and superalgebra such that when considered as modules of $\mathcal{U}(\fh)$ they are free of rank 1. We provide sufficient conditions for simplicity, and compute the action of an automorphism on the module. Lastly, we apply the weighting functor and show that the resulting modules are intermediate series modules.   
\end{abstract}

\tableofcontents

%%%%%%%%%%%%%%%%%%%%%%%%%%%%%%%%%%%%%%%%%%%%%%%%%%%%%%%%%%%%%%%%%%%%%%%%%%%%%%%%%%%%%

\section{Introduction} 

%%%%%%%%%%%%%%%%%%%%%%%%%%%%%%%%%%%%%%%%%%%%%%%%%%%%%%%%%%%%%%%%%%%%%%%%%%%%%%%%%%%%%

In this thesis we consider a Witt-type Lie algebra $\mathcal{W}_\Gamma$ for $\Gamma$ a subgroup of $\CC$ containing 1, which we will call the $\Gamma$-Witt algebra, and its super-analogue $s\mathcal{W}_\Gamma$, the $\Gamma$-Witt superalgebra. We will refer to the classic Witt algebra with $\mathcal{W}_\ZZ$. The centerless Ramond and Neveu-Schwarz Lie superalgebras are special cases of $s\mathcal{W}_\Gamma$, and we will denote them with $s\mathcal{W}_\ZZ$ and $s\mathcal{W}_{\ZZ + 1/2}$, respectively. The Witt-type Lie algebra $\mathcal{W}_\QQ$ was introduced by Mazorchuk in \cite{mazorchuk1997classification}. Naturally we have the central extensions, which we will call the $\Gamma$-Virasoro algebra and superalgebra, with notation $\mathcal{V}_\Gamma$ and $s\mathcal{V}_\Gamma$. A subalgebra spanned by $L_0$ is a Cartan subalgebra of $\mathcal{W}_\Gamma$; we denote this subalgebra with $\fh$. The algebra $s\mathcal{W}_\Gamma$ admits a Cartan subalgebra $s\fh$ containing $L_0$. 
%Let $\fg$ be $\mathcal{W}_\Gamma$ or $s\mathcal{W}_\Gamma$ and let $\fh$ be the Cartan subalgebra of $\fg$. 

 Intermediate series modules were introduced as part of V. G. Kac's conjecture regarding Harish-Chandra modules of Virasoro. V. G. Kac conjectured that any Harish-Chandra module over the Virasoro algebra was either highest weight, lowest weight, or a module of the intermediate series \cite{kac1979structure, kac1982some}. This theorem was proved by O. Mathieu in \cite{Mathieu1992} and we use the definition of an intermediate series module provided in his paper. Subsequently, the analogous superalgebra conjecture was proven by Y. Su in \cite{Yucai1995}. Each family of intermediate series modules over $\mathcal{W}_\ZZ$ (resp. $s\mathcal{W}_\ZZ$, $s\mathcal{W}_{\ZZ + 1/2}$) becomes a $\mathcal{W}_\Gamma$--module (resp. $s\mathcal{W}_\Gamma$--module), which we will call an intermediate series module over $\mathcal{W}_\Gamma$ (resp. $s\mathcal{W}_\Gamma$) and denote with $V_\CC(a)$ (resp. $sV_\CC(a)$). 
 
 Let $\fh \subset \fg$ be a Lie algebra such that $\Ad(\fh)$ acts diagonally on $\fg$. In \cite{Nilsson2016}, following a suggestion from O. Mathieu, J. Nilsson introduced a weighting functor $W$. This is a covariant right exact functor from the category of $\fg$--modules to the category of weight $\fg$-modules. On a weight module, $W$ acts by identity. Let $\mathcal{M}$ be the category of $\fg$--modules which are free of rank 1 as a $\mathcal{U}(\fh)$--module. For each $M \in \mathcal{M}$ the weight spaces of the module $W(M)$ are all one-dimensional; for a semisimple Lie algebra $\fg$ the modules $W(M)$ give coherent families. In this thesis we describe $\mathcal{M}$ for Witt-type Lie (super)algebras. 
 
 The main results of the thesis are the following. 

In \Cref{automorphism group section} we compute the group of automorphisms of $\mathcal{W}_\Gamma$ and show that the automorphisms preserve $\fh$. 

In \Cref{modules section} we classify $\mathcal{W}_\Gamma$--modules (resp. $s\mathcal{W}_\Gamma$--modules) which when restricted to $\mathcal{U}(\fh)$ are free of rank 1. These modules are of the form $\Omega_\Gamma(f, \alpha)$ (resp. $s\Omega_\Gamma(f, \alpha)$) where $\alpha \in \CC$, and $f: \Gamma \ra \CC^*$ is a group homomorphism. The modules $\Omega_\Gamma(f, \alpha), s\Omega_\Gamma(f, \alpha)$ are simple for $\alpha \neq 0$; there are monomorphisms
$$\Omega_\Gamma(f, 1)\to\Omega_\Gamma(f,0),\ \ \ \ \ 
\Pi(s\Omega_\Gamma(f, 1/2))\to s\Omega_\Gamma(f,0)$$
whose cokernels are the trivial representations. 
 We compute the action of automorphisms on $\Omega_\Gamma(f, \alpha)$ and $s\Omega_\Gamma(f, \alpha)$. The restriction of $\Omega_\Gamma(f, \alpha)$ (resp. $s\Omega_\Gamma(f, \alpha)$) to $\mathcal{W}_\ZZ$ (resp. $s\mathcal{W}_\ZZ$ and $s\mathcal{W}_{\ZZ + 1/2}$) are the modules $\Omega_\ZZ(\lambda, \alpha)$ (resp. $s\Omega_\ZZ(\lambda, \alpha)$ and $s\Omega_{\ZZ+1/2}(\lambda, \alpha)$), researched in \cite{lu2014irreducible, tan2015wn+, chen2018non, Yang2019} and others. 

Lastly, in \Cref{weighting functor section} we consider the weighting functor. We apply the weighting functor to the modules $\Omega_\Gamma(f, \alpha)$ (resp. $s\Omega_\Gamma(f, \alpha)$) and show that the resulting modules are of the form $V_\CC(1-\alpha)$ (resp. $sV_\CC(1-\alpha)$). Moreover, we show that as a $\mathcal{W}_\CC$--module (resp. $s\mathcal{W}_\CC$--module), $V_\CC(\alpha)$ (resp. $sV_\CC(\alpha)$) is simple for $\alpha \neq 0, 1$ (resp. $\alpha \neq 1/2, 1$).

%%%%%%%%%%%%%%%%%%%%%%%%%%%%%%%%%%%%%%%%%%%%%%%%%%%%%%%%%%%%%%%%%%%%%%%%%%%%%%%%%%

\section{Preliminaries}

%%%%%%%%%%%%%%%%%%%%%%%%%%%%%%%%%%%%%%%%%%%%%%%%%%%%%%%%%%%%%%%%%%%%%%%%%%%%%%%%%%%

Our base field is $\CC$. Throughout this thesis $\ZZ, \; \QQ, \; \CC, \;\CC^*$ will denote the integers, rationals, complex numbers, and non-zero complex numbers respectively. For a Lie algebra $\fg$ we use $\mathcal{U}(\fg)$ to denote the \textit{universal enveloping algebra.} For a super vector space $V = V_{\overline{0}} \oplus V_{\overline{1}}$, we use $\Pi$ to denote the \textit{parity switching functor}: $\Pi(V)_{\overline{i}} \defeq V_{\overline{i + 1}}$. This gives a parity switching functor for modules over Lie superalgebras. 

For a subgroup $\Gamma \subset \CC$ containing 1 we consider the \textit{$\Gamma$-Witt algebra}. This generalization of the Witt algebra for $\Gamma = \QQ$ was introduced by Mazorchuk in \cite{mazorchuk1997classification}. Define $\mathcal{W}_\Gamma$ to be a Lie algebra with basis $\{L_m\}_{m \in \Gamma}$ and the commutation relations
\begin{equation}\label{Wittbracket}
    [L_m, L_n] = (m-n)L_{m+n}.
\end{equation}
The Lie algebra $\mathcal{W}_\Gamma$ contains the classical Witt algebra $\mathcal{W}_\ZZ$, and the Cartan subalgebra $\fh = \CC L_0$. The central extension of $\mathcal{W}_\Gamma$ is the \textit{$\Gamma$-Virasoro algebra}, which we will denote with $\mathcal{V}_\Gamma$, with the commutation relations
\begin{equation}\label{Virasorobracket}
    [L_m, C] = 0, \; \; [L_m, L_n] = (m-n)L_{m+n} + \delta_{m + n, 0} \frac{m^3 - m}{12}C 
\end{equation}

In preparation to defining the analogous superalgebra, let $\Gamma$ be a subgroup of $\CC \times \ZZ_2$. We set 
\[
\Gamma_0 \defeq \{x \in \CC \mid (x, 0) \in \Gamma\}, \; \; \Gamma_1 \defeq \{x \in \CC \mid (x, 1) \in \Gamma\}.
\]
We note that $\Gamma_0$ is a subgroup of $\CC$, $\Gamma_1 + \Gamma_1 \subset \Gamma_0$ and $\Gamma_1 + \Gamma_0 \subset \Gamma_1$. With this in hand we define the $\Gamma$-\textit{Witt superalgebra}. Let 
\[
s\mathcal{W}_\Gamma = (s\mathcal{W}_\Gamma)_{\overline{0}} \oplus (s\mathcal{W}_\Gamma)_{\overline{1}}
\]
be a Lie superalgebra with basis $\{L_m\}_{m \in \Gamma_0}$ in $(s\mathcal{W}_\Gamma)_{\overline{0}}$, $\{G_r\}_{r \in \Gamma_1}$ in $(s\mathcal{W}_\Gamma)_{\overline{1}}$ and commutation relations \eqref{Wittbracket}, 
\begin{equation}
    [G_r, G_s] = 2L_{r+s}, \; \;  
    [L_m, G_r] = \left(\frac{m}{2} - r \right) G_{m+r}.
\end{equation}
One has that 
\[
(s\mathcal{W}_\Gamma)_{\overline{0}} = \mathcal{W}_{\Gamma_0}. 
\]
In $s\mathcal{W}_\Gamma$ we contain the subalgebras \textit{centerless Ramond} for $\Gamma = \ZZ \times \ZZ_2$ and \textit{centerless Neveu-Schwarz} for $\Gamma_0 = \ZZ, \Gamma_1 = \ZZ + 1/2$, which will refer to with $s\mathcal{W}_\ZZ$ and $s\mathcal{W}_{\ZZ + 1/2}$ respectively. Furthermore, the space $s\fh = \CC L_0 \oplus \CC G_0$ (resp. $s\fh = \CC L_0$) with dimension $(1 \mid 1)$ (resp. dimension $(1 \mid 0)$) if $0 \in \Gamma_1$ (resp. $0 \notin \Gamma_1$) is a Cartan subalgebra of $s\mathcal{W}_\Gamma$.  

The central extension to $s\mathcal{W}_\Gamma$ is the $\Gamma$-\textit{Virasoro superalgebra}. Denote this with \[
s\mathcal{V}_\Gamma = (s\mathcal{V}_\Gamma)_{\overline{0}} \oplus (s\mathcal{V}_\Gamma)_{\overline{1}},
\] with basis $\{L_m, C\}_{m \in \Gamma_0}$ in $(s\mathcal{V}_\Gamma)_{\overline{0}}$, $\{G_r\}_{r \in \Gamma_1}$ in $(s\mathcal{V}_\Gamma)_{\overline{1}}$, and commutation relations \eqref{Virasorobracket},
\begin{align*}
    [L_m, G_r] = \left( \frac{m}{2} - r\right) G_{m + r}, \; \;  [G_r, G_s] = 2L_{r + s} + \frac{C}{3}\left(r^2 - \frac{1}{4}\right) \delta_{r+s, 0}
\end{align*}
and $[C, G_r] = 0$. Within $s\mathcal{V}_\Gamma$ we contain the well-known subalgebras \textit{Ramond} for $\Gamma_0 = \Gamma_1 = \ZZ$ and \textit{Neveu--Schwarz} for $\Gamma_0 = \ZZ, \Gamma_1 = \ZZ + 1/2$.

%%%%%%%%%%%%%%%%%%%%%%%%%%%%%%%%%%%%%%%%%%%%%%%%%%%%%%%%%%%%%%%%%%%%%%%%%%%%%%%%%%%%%%%%%%%%%%%%%%%%%%%%%%%%%%%%%%%%

\section{Automorphisms}\label{automorphism group section}

%%%%%%%%%%%%%%%%%%%%%%%%%%%%%%%%%%%%%%%%%%%%%%%%%%%%%%%%%%%%%%%%%%%%%%%%%%%%%%%%%%%%%%%%%%%%%%%%%%%%%%%%%%%%%%%%%%%%

Denote $\Gamma^{\lor}$ as the set of homomorphisms $\chi: \Gamma \ra \CC^*$ (i.e. the maps $\chi$ that satisfy $\chi(x_1 + x_2) = \chi(x_1)\chi(x_2)$ for each $x_1, x_2 \in \Gamma$.) Note that $\Gamma^{\lor}$ is an abelian group with the operation $(\chi_1\chi_2)(x) \defeq \chi_1(x)\chi_2(x)$. In this section we compute the automorphism groups of $\mathcal{W}_\Gamma$ (resp. $s\mathcal{W}_\Gamma$) and show that they preserve $\fh$ (resp. $s\fh$). 

\begin{lemma}\label{actingdiagonally}
    The only elements of $\mathcal{W}_\Gamma$ that act diagonally in the adjoint representation lie in $\CC L_0$.
\end{lemma}
    \begin{proof}
        For this proof, we impose the following order on elements of $\CC$: $a + bi > c + id$ if and only if $a > c$ or $a = c, b > d$. It is straightforward to check that, given four elements $a_1, a_2, b_1, b_2 \in \CC$ such that $a_1 < a_2,\; b_1 \leq b_2$, then $a_1 + a_2 < b_1 + b_2$. 
        
        For an element $\sum_{a\in \CC} x_a L_a \in \mathcal{W}_{\Gamma}$ 
        we set 
        $$\supp \left(\sum_{a\in \Gamma} x_a L_a\right):=\{a\in\CC|\ x_a \neq 0\}.$$
        For each non-zero $A, B \in  \mathcal{W}_{\Gamma}$ one has
        $$\max(\supp(A)) \neq \max(\supp(B)) \implies  \max(\supp([A,B]))=\max(\supp(A))+\max(\supp(B)).$$
        Let $A \in  \mathcal{W}_{\Gamma}$ be a non-zero element 
        acting diagonally in the adjoint representation.
        If $\max (\supp (A)) \neq 0$, then, by above, any eigenvector $B$
        of $\ad (A)$ satisfies 
        \[
        \max(\supp(B)) = \max(\supp(A)),
        \]
        so
        the eigenvectors do not span $\mathcal{W}_{\Gamma}$.
        Hence  $\max (\supp (A)) = 0$. Similarly,  $\min (\supp (A)) = 0$. Therefore,
        $\supp(A)= \{0\}$ and $A \in \CC L_0$.
    \end{proof}

\begin{theorem}\label{witautomorphisms}
\begin{enumerate}[i.]
    \item $\aut(\mathcal{W}_\Gamma) = \{\varphi_{a, f} \mid f \in \Gamma^{\lor}, \; a \in \CC^* \text{ s.t. } a\Gamma = \Gamma \}$ where
    \[
    \varphi_{a, f}(L_m) = a f(m) L_{ m/a}.
    \]
    
    \item $\aut(s\mathcal{W}_\Gamma) = \{\varphi_{a, f} \mid  f \in \Gamma^{\lor}, \; a \in \CC^* \suchthat a^2\Gamma = \Gamma, \; \text{Re}(a) \geq 0\}$ where 
    \[
    \varphi_{a, f}(L_m) = a^2 f(m) L_{m/a^2}, \; \; \varphi_{a,f}(G_r) = a f(r) G_{r/a^2}.
    \]
\end{enumerate}
\end{theorem}
\begin{proof}
     (i) Let $\varphi \in \Aut(\mathcal{W}_\Gamma)$. By \Cref{actingdiagonally} we must have that $\varphi(L_0) = aL_0$ for $a \in \CC^*$. As $\spec(\Ad(L_0)) = \Gamma$ we obtain $a\Gamma = \Gamma$. By \eqref{Wittbracket} for each $m \in \Gamma$ we have that for some $b_m \in \CC^*$,
    \[
    \varphi(L_m) = b_m L_{m/a}
    \]
    with $(m - n)b_{m + n} = (1/a)(m - n)b_m b_n$. Let us define $f(z) \defeq b_z/a$. Then $f: \Gamma \ra \CC^*$ has the properties $f(0) = 1$ and  
    \begin{equation}\label{functional equation for automorphisms}
        f(z_1 + z_2) = f(z_1)f(z_2) 
    \end{equation}
    for $z_1 \neq z_2$. Using the fact that $f(0) = 1$ we find 
    \[
    1 = f(z + (-z)) = f(z) f(-z) \implies f(-z) = f(z)^{-1}.
    \]
    With this in hand we see
    \[
    f(z) = f(2z + (-z)) = f(2z)f(-z)  \implies f(2z) = f(z)^2,
    \]
    which implies \eqref{functional equation for automorphisms} holds for all $z_1, z_2 \in \Gamma$. Hence, $f \in \Gamma^{\lor}$ as required. \\ \\
    (ii) Let $\varphi \in \aut(s\mathcal{W}_\Gamma)$. By (i) we know 
    \[
    \varphi(L_m) = a_0f(m)L_{m/a_0}
    \]
    for $f \in \Gamma_0^\lor$, $a_0 \in \CC^*$ such that $a_0 \Gamma = \Gamma$. Upon consideration of the commutator $[L_0, G_r] = rG_r$ we must have that for some $b_r \in \CC^*$,
    $$
    \varphi(G_r) = b_r G_{r/a_0}.
    $$
    Take $a \in \CC^*$ such that $a^2 = a_0$ and $\text{Re}(a) \geq 0$. Define a function $g: \Gamma_1 \ra \CC^*$ by $g(r) \defeq b_r/a$. By the commutation relations $[L_m, G_r]$ and $[G_r, G_s]$ one finds the properties of the function $g: \Gamma_1 \ra \CC$
    \begin{equation}\label{superautomorphisms}
         f(m) g(r)  = g(m + r), \; \; g(r) g(s) = f(r + s), \; m \in \Gamma_0, \; r, s \in \Gamma_1
    \end{equation}
    respectively. We may define the function $f'$ as $f'(m) = f(m)$ for $m \in \Gamma_0$, and $f'(r) = g(r)$ for $r \in \Gamma_1$. By \eqref{superautomorphisms} we necessarily have that $f' \in \Gamma^{\lor}$, and thus conclude the theorem. 
\end{proof}

%%%%%%%%%%%%%%%%%%%%%%%%%%%%%%%%%%%%%%%%%%%%%%%%%%%%%%%%%%%%%%%%%%%%%%%%%%%%%%%%%%%%%%%%

\section{Modules}\label{modules section}

%%%%%%%%%%%%%%%%%%%%%%%%%%%%%%%%%%%%%%%%%%%%%%%%%%%%%%%%%%%%%%%%%%%%%%%%%%%%%%%%%%%%%%%%%

In this section we classify $\mathcal{W}_\Gamma$--modules (resp. $s\mathcal{W}_\Gamma$--modules) which when restricted to $\mathcal{U}(\fh)$ are free of rank 1. 

\begin{definition}
    Fix $\alpha \in \CC$, $f \in \Gamma^{\lor}$. View $\CC[L_0]$ as a $\mathcal{W}_\Gamma$--module with action 
    \[
    L_m P(L_0) = f(m)(L_0 + m \alpha)P(L_0 + m), \; \; P(L_0) \in \CC[L_0].
    \]
    We denote this module by $\Omega_\Gamma(f, \alpha)$. 
\end{definition}
For the classical Witt algebra $\mathcal{W}_\ZZ$, these modules were introduced in \cite{lu2014irreducible} and it was shown in \cite{tan2015wn+} that they constitute the only free rank 1 $\mathcal{U}(\fh)$--modules of $\mathcal{W}_\ZZ$. 

Next, we define the analogous modules of $s\mathcal{W}_\Gamma$. For $s\mathcal{W}_\ZZ$ (resp. $s\mathcal{W}_{\ZZ + 1/2}$) these modules were introduced in \cite{Yang2019}, and it was shown that they constitute the only $s\mathcal{W}_\ZZ$--modules (resp. $\mathcal{W}_{\ZZ + 1/2}$) such that the restriction on $\mathcal{U}(s\fh)$ is a free module of rank 1 (resp. $(1 \mid 1)$).
\begin{definition}
Fix $\alpha \in \CC$, $f \in \Gamma^{\lor}$. View $\CC[L_0] \oplus \xi \CC[L_0]$ as a $s\mathcal{W}_\Gamma$--module with actions,
\begin{align*}
    L_m(P(L_0) + \xi Q(L_0)) &= f(m)(L_0 + m\alpha)P(L_0 + m) + f(m)(L_0 + m(\alpha + 1/2))\xi Q(L_0 + m) \\
    G_r(P(L_0) + \xi Q(L_0)) &= f(r)(L_0 + 2r\alpha) Q(L_0 + r) + f(r)\xi P(L_0  + r)
\end{align*}
for $P(L_0) + \xi Q(L_0) \in \CC[L_0] \oplus \xi \CC[L_0]$. We denote this module by $s\Omega_\Gamma(f, \alpha)$. Note that  
\[
(s\Omega_\Gamma)_{\overline{0}} \cong \Omega_{\Gamma_0}(f_0, \alpha) \; \; \text{  and  } \; \; (s\Omega_\Gamma)_{\overline{1}} \cong \Omega_{\Gamma_0}(f_0, \alpha + 1/2)
\]
where $f_0$ is the restriction of $f$ to $\Gamma_0$.  
\end{definition}
$\Omega_\Gamma(f, \alpha)$ is defined as a $\mathcal{V}_\Gamma$--module with zero action of $C$, and similarly for $s\Omega_\Gamma(f, \alpha)$. 

%%%%%%%%%%%%%%%%%%%%%%%%%%%%%%%%%%%%%%%%%%%%%%%%%%%%%%%%%%%%%%%%%%%%%%%%%%%%%%%%%%%%%%%%%%%%

\subsection{Theorem}\label{provingCactsbyzero}

%%%%%%%%%%%%%%%%%%%%%%%%%%%%%%%%%%%%%%%%%%%%%%%%%%%%%%%%%%%%%%%%%%%%%%%%%%%%%%%%%%%%%%%%%%%

\begin{proposition}\label{Omegaabaredifferent}
    For $f \neq f', \alpha \neq \alpha'$, $\Omega_\Gamma(f,\alpha) \ncong \Omega_\Gamma(f', \alpha')$. 
\end{proposition}
\begin{proof}
    Assume that 
    \[
    \varphi: \Omega_\Gamma(f,\alpha) \ra \Omega_\Gamma(f',\alpha')
    \]
    is an isomorphism. Denote $1$ and $1'$ as the generators of $\CC[L_0]$--modules $\Omega_\Gamma(f, \alpha)$ and $\Omega_\Gamma(f', \alpha')$ respectively. As $\varphi$ is an isomorphism, we must have that $\varphi(1) = c1'$ for some $c \in \CC^*$. Then, $\varphi(L_01) = cL_01'$ and by induction $\varphi(P(L_0)) = cP(L_0)$. Therefore $f = f',$ $\alpha = \alpha'$ and we conclude the proposition.
\end{proof} 

\begin{corollaryy}
    \begin{enumerate}[i.]
    
        \item For $\Gamma \subset \CC \times \ZZ_2$, let $\phi: \Gamma \ra \ZZ_2$ be an element of $\Gamma^{\lor}$ defined by $\phi(a) = (-1)^{i}$ for $a \in \Gamma_{i}$. Then $s\Omega_\Gamma(f, \alpha) \cong s\Omega_\Gamma(f', \alpha')$ if and only if $\alpha' = \alpha$ and $f' = f$ or $f' = f \phi$. 
        
        \item $\Pi(s\Omega_\Gamma(f, \alpha))$ is not isomorphic to $s\Omega_\Gamma(f', \alpha')$ for any $f' \in \Gamma^{\lor}$, $\alpha' \in \CC$. 
        
    \end{enumerate} 
\end{corollaryy}
\begin{proof}
    (i) Assume 
    $
    \varphi: s\Omega_\Gamma(f, \alpha) \ra s\Omega_\Gamma(f', \alpha')
    $
    is an isomorphism, and let us denote 
    \[
    f_0 \defeq f\big|_{\Gamma_0}, \; \; \; f'_0 \defeq f'\big|_{\Gamma_0}. 
    \] 
    We recall that 
    \[
    s\Omega_\Gamma(f, \alpha)_{\overline{0}} \cong \Omega_{\Gamma_0}(f_0, \alpha) \; \text{ and } \; s\Omega_\Gamma(f', \alpha')_{\overline{0}} \cong \Omega_{\Gamma_0}(f'_0, \alpha' + 1/2).
    \]
    % as well as 
    % \[
    % s\Omega_\Gamma(f', \alpha')_{\overline{0}} \cong \Omega_{\Gamma_0}(f'_0, \alpha') \; \text{ and } \; s\Omega_\Gamma(f', \alpha')_{\overline{1}} \cong \Omega_{\Gamma_0}(f'_0, \alpha' + 1/2). 
    % \]
    Then by \Cref{Omegaabaredifferent} we immediately see that $\alpha = \alpha'$. Let $1$, $1'$, $\xi$, $\xi'$ be the generators of the free $\mathcal{U}(\fh)$--modules 
    \[
    s\Omega_\Gamma(f, \alpha)_{\overline{0}}, \; \;  s\Omega_\Gamma(f', \alpha)_{\overline{0}}, \; \; s\Omega_\Gamma(f, \alpha)_{\overline{1}},\; \;  s\Omega_\Gamma(f', \alpha)_{\overline{1}}
    \]
    respectively. As we assumed $s\Omega_\Gamma(f, \alpha) \cong s\Omega_\Gamma(f', \alpha)$ we necessarily have that 
    \[
    \varphi(1) = a_1 1', \; \; \; \varphi(\xi) = a_{\xi} \xi'
    \]
    for some $a_1, a_\xi \in \CC^*$. Then 
    \begin{align*}
        a_1 f'(r) \xi'= G_r\varphi(1) &= \varphi(G_r 1) = a_\xi f(r)\xi' \\
        a_\xi f'(r)(L_0 + 2r\alpha) 1' = G_r\varphi(\xi) &= \varphi(G_r \xi) = a_1 f(r)(L_0 + 2r\alpha)1'.
    \end{align*}
    This tells us
    \[
    a_1f'(r) = a_\xi f(r), \; \;  a_\xi f'(r) = a_1 f(r)
    \]
    which together implies $a_1^2 = a_\xi^2$. If $a_1 = a_\xi$ then $f(r) = f'(r)$, and if $a_1 = -a_\xi$, then 
    \[
    f'(r) = -f(r) = f(r)\phi(r).
    \]
    This concludes (i). 
    \\
    \\
    (ii) Assume 
    \[
    \varphi: \Pi(s\Omega_\Gamma(f, \alpha)) \ra s\Omega_\Gamma(f', \alpha')
    \]
    is an isomorphism. Let $1, \; 1', \; \xi, \; \xi'$ be the generators of the free $\CC[L_0]$--modules \[
    s\Omega_\Gamma(f, \alpha)_{\overline{0}},\; \; s\Omega_\Gamma(f', \alpha')_{\overline{0}}, \; \; s\Omega_\Gamma(f, \alpha)_{\overline{1}}, \; \; s\Omega_\Gamma(f', \alpha')_{\overline{1}}
    \]
    respectively. Then as $\varphi$ is an isomorphism we must have that 
    \[
    \varphi(1) = c_1 \xi', \; \; \varphi(\xi) = c_2 1'
    \]
    for some $c_1, c_2 \in \CC^*$. However, notice
    \[
    \varphi(G_r \xi) = \varphi(f(r)(L_0 + 2r\alpha)1) = f(r)(L_0 + 2r\alpha)c_1\xi', \; \; 
    G_r\varphi(\xi) = c_2 f'(r) \xi',
    \]
    so $\varphi(G_r \xi) \neq G_r \varphi(\xi)$, a contradiction. This concludes (ii).
\end{proof}

\begin{theorem}\label{maintheorem1}
    Let $M$ be a $\mathcal{V}_\Gamma$-module. If the restriction of $M$ to $\CC[L_0]$ is a free module of rank 1, then $C$ acts by zero and $M = \Omega_\Gamma(f, \alpha)$ for some $f \in \Gamma^{\lor}$, $\alpha \in \CC$. 
\end{theorem}
\begin{proof}
    Let $M$ be a $\mathcal{V}_\Gamma$ module as described, and identify $M$ with $\CC[L_0]$ as a vector space. For each $m \in \Gamma$ let us define 
    \[
     L_m 1 \defeq q_m(L_0).
    \]
    Through \eqref{Virasorobracket} we obtain 
    $
    L_mL_0 P(L_0) = (L_0 + m) L_m P(L_0),
    $
    which in particular gives us the recursive relation
\begin{equation}\label{second}
    L_mL_0^j = (L_0 + m)^j q_m(L_0).
\end{equation}
    From \eqref{second} it follows that 
\begin{equation}
    L_mP(L_0) = q_m(L_0)P(L_0 + m), \; \; \forall P(L_0) \in \CC[L_0]. 
\end{equation}
    Next, let us show that $C$ must act by a constant. Define 
    \[
    P_C(L_0) \defeq C 1,
    \]
    and note that as $C$ is in the center we have 
    $
    C L_0 = C(L_0 1) = L_0 P_C(L_0).
    $
    By induction this tells us $C L_0^m = L_0^mP_C(L_0)$. So,
    $$
    CP(L_0) = P(L_0)P_C(L_0), \; \; \forall P(L_0) \in \CC[L_0].  
    $$
    Let us fix a polynomial $P(L_0) \in \CC[L_0]$. We notice that $L_m (C P(L_0)) = C(L_m P(L_0))$ implies 
    $$
    q_m(L_0)P(L_0+m) P_C(L_0 + m) = q_m(L_0) P(L_0 + m) P_C(L_0)$$ 
    so $P_C(L_0) = K$ for some constant $K$. With this in hand, we can show that $q_m(L_0)$ is a linear polynomial. Notice that from \eqref{Virasorobracket} we find the relation
    \begin{equation}\label{relation on qm}
    q_m(L_0) q_n(L_0 + m) - q_n(L_0)q_m(L_0 + n) = (m-n)q_{m + n}(L_0) + \frac{m^3 - m}{12}K \delta_{m + n, 0}.
    \end{equation}
    Define $Q(L_0) \defeq q_m(L_0)q_{-m}(L_0 + m)$. Then by \eqref{relation on qm} we have 
    $$ Q(L_0) - Q(L_0 - m) = 2mL_0 + \frac{(m^3-m)K}{12}$$ 
    showing that $Q(L_0)$ is at most quadratic, and so $q_m(L_0)$ is at most quadratic. Assume for the sake of contradiction that for some $m$, $q_m(L_0)$ is a constant polynomial. Then from \eqref{relation on qm} either $\deg(q_{m + n}) < \deg(q_n)$ for all $n \neq \pm m$ or $q_{m+n} = 0$. In particular, we see $\deg(q_{3m + 1}) < \deg(q_{2m + 1}) < \deg(q_{m + 1}) < \deg(q_1) \leq 2$. However, this implies that $q_{3m + 1} = 0$ and thus $q_j = 0$ for all $j$, a contradiction. Therefore, it cannot be the case that $q_m(L_0)$ was a constant, and we conclude that $q_m(L_0)$ is a degree one polynomial, i.e.
    \[
    q_m(L_0) = \lambda_m(L_0 + m\alpha_m)
    \]
    for $\lambda_m \in \CC^*, \alpha_m \in \CC$ with $\lambda_0 = 1, \alpha_0 = 0$. Substituting this into \eqref{relation on qm} we find 
    \begin{equation}\label{new relation on qm}
    \lambda_m \lambda_n ((m-n)L_0 + m^2 \alpha_m - n^2 \alpha_n) = (m-n)\lambda_{m + n}(L_0 + (m+n)\alpha_{m + n}) + \frac{m^3 - m}{12}K \delta_{m + n, 0}.    
    \end{equation}
    By \eqref{new relation on qm} we have the following relation on $\lambda_m$,
    \begin{equation*}\label{lambda relation}
        (m-n)\lambda_{m + n} = (m -n)\lambda_m\lambda_n,
    \end{equation*}
    and we may argue as before in the proof of \Cref{witautomorphisms} to find that $f(m) \defeq \lambda_m \in \Gamma^{\lor}$, as required. We will now show that $K = 0$. Through \eqref{new relation on qm}, one can easily find that
    \[
    \alpha_{-1} = \alpha_1, \; \alpha_2 =  \alpha_1, \; \alpha_{-2} = \alpha_1
    \]
    by taking the values $m = 1, 2, -2$ with $n = -1, -1, 1$ respectively. Then let $m = 2, n = -2$ in \eqref{new relation on qm} to conclude
    $
    K = 0.
    $
    So by \eqref{new relation on qm} we must have the following relation on $\alpha_m$, 
    \begin{equation}\label{alpha relation}
        m^2 \alpha_m - n^2 \alpha_n = (m^2 - n^2)\alpha_{m +n}.
    \end{equation}
    Taking $m = -n$ in \eqref{alpha relation} we find $\alpha_n = \alpha_{-n}$. Next, using this fact and choosing values $m+n, -n$ in \eqref{alpha relation} we see that 
    \begin{equation}\label{second alpha relation}
    (m+n)^2 \alpha_{m + n} - n^2\alpha_{-n} = (m^2 + 2mn)\alpha_m.
    \end{equation}
    Combining \eqref{second alpha relation} and \eqref{alpha relation} we conclude $\alpha_m = \alpha_n$. Therefore, $q_m(L_0) = f(m)(L_0 + \alpha_1 m)$ and $M \cong \Omega_\Gamma(f, \alpha_1)$. 
\end{proof}

% \begin{corollary}\label{witt c modules}
%     Let $M$ be a $\mathcal{W}_\Gamma$--module such that $L_0$ acts by multiplication. Then $M = \Omega_\Gamma(f, \alpha)$ for $\alpha \in \CC$, $f \in \Gamma^{\lor}$. 
% \end{corollary}
% \begin{proof}
%     As we have shown that such a module of $\mathcal{V}_\Gamma$ requires that $C$ acts by zero, we naturally have the structure of a $\mathcal{W}_\Gamma$--module.
% \end{proof}

\begin{theorem} 
Let $M = M_{\overline{0}} \oplus M_{\overline{1}}$ be a $s\mathcal{V}_\Gamma$-module. If the restriction of $M$ to $\CC[L_0] \oplus \xi\CC[L_0]$ is free of rank $(1 \mid 1)$, then $C$ acts by zero and 
\[
M \cong s\Omega_\Gamma(f, \alpha) \text{ or } M \cong \Pi(s\Omega_\Gamma(f, \alpha))
\]
\end{theorem}
\begin{proof}
Let $M$ be a module as described above. As both $M_{\overline{0}}$ and $M_{\overline{1}}$ are necessarily modules of $\mathcal{V}_\Gamma$, we know that $C$ must act by zero and 
\[
M_{\overline{0}} = \Omega_{\Gamma_0}(f, \alpha),\; \;  M_{\overline{1}} = \Omega_{\Gamma_0}(f', \alpha')
\]
for some $\alpha,$ $ \alpha' \in \CC$, $f,$ $ f' \in \Gamma_{0}^{\lor}$. Denote by 1 the generator of a free $\CC[L_0]$--module $M_{\overline{0}}$, and by $\xi$ the generator of a free $\CC[L_0]$--module $M_{\overline{1}}$. For each $r \in \Gamma_1$ define 
\[
G_r1 \defeq p_r(L_0) \xi, \; \; G_r \xi \defeq q_r(L_0) 1. 
\]
Using the commutator $[L_0, G_r]$ one can easily show that 
\[
G_r(P(L_0) + \xi Q(L_0)) = q_r(L_0) Q(L_0 + r) + p_r(L_0) \xi P(L_0 + r).
\]
We note that by substituting $\xi$ with $\xi' \defeq c \xi$ for $c \in \CC^*$, we obtain $p'_r(L_0) = p_r(L_0)/c$ and $q'_r(L_0) = cq_r(L_0)$. Let us choose $c$ in such a way that the polynomials $p_{r_1}(L_0)$ and $ q_{r_1}(L_0)$ have the same leading coefficient for some specific $r_1 \in \Gamma_1$. \\
\\
Next, fix polynomials $P(L_0) \in M_{\overline{0}}$, $Q(L_0) \in M_{\overline{1}}$. Using $[G_r, G_s]P(L_0) = 2L_{r+s}P(L_0)$ we find the relations
\begin{align}
q_r(L_0) p_s(L_0 + r) + q_s(L_0)p_r(L_0 + s) &= 2f(r+s)(L_0 + (r+s)\alpha), \label{randsrelationshipanticommutator} \\
 q_r(L_0)p_r(L_0 + r) &= f(2r)(L_0 + 2r\alpha). \label{r=srelationshipanticommutator}  
\end{align}
and the commutators $[L_m, G_r] = (m/2 - r)G_{m+r}$ give the relations
\begin{align}
p_r(L_0)f'(m)(L_0 + m\alpha') - p_r(L_0) f(m)(L_0 + r + m\alpha) &= \left(\frac{m}{2} - r \right) p_{m+r}(L_0), \label{evengrlm}\\
q_r(L_0)f(m)(L_0 + m \alpha) - q_r(L_0) f'(m)(L_0 + r + \alpha' m) &= \left(\frac{m}{2} - r\right) q_{r + m}(L_0) \label{svone}
\end{align}
for $P(L_0)$ and $\xi Q(L_0)$, respectively. From \eqref{svone} we immediately see that the coefficient of the linear term  must equal zero, giving us
$
f'(m) = f(m), \; \forall \; m \in \Gamma_0. 
$ 
Furthermore, by \eqref{r=srelationshipanticommutator} it must be the case that one of $q_r(L_0),\; p_r(L_0)$ is linear and the other is a constant. By \eqref{randsrelationshipanticommutator} we know that this choice must hold for each $r \in \Gamma_1$. Let us first choose
\begin{equation}\label{qrandpr}
p_r(L_0) = c_r,~q_r(L_0) = a_r(L_0 + b_r),\; \; a_r, c_r \in \CC^*, \; b_r \in \CC.
\end{equation}
Substituting \eqref{qrandpr} into \eqref{r=srelationshipanticommutator} we find that $c_r = f(2r)/a_r$, $b_r = 2 r \alpha$ for all $r \in \Gamma_1$. This allows us to re-write \eqref{qrandpr} as 
\begin{equation}\label{newqrandpr}
p_r(L_0) = \frac{f(2r)}{a_r}, \; \; q_r(L_0) = a_r(L_0 + 2 r \alpha). 
\end{equation}
Let us combine formulae \eqref{newqrandpr} and \eqref{svone}. Through the linear term we find
\begin{equation*}
A = f(m)a_r(m\alpha + m - m\alpha' - r) = \left(\frac{m}{2} - r \right) a_{m + r} = B
\end{equation*}
and for the constant term 
\begin{equation*}
C = f(m)a_r(m^2\alpha + 2rm\alpha^2) - f(m)a_r(2r\alpha\alpha' m + 2r^2\alpha) = \left(\frac{m}{2} - r \right)a_{m + r}(2(m+r)\alpha) = D 
\end{equation*}
where we have added the notation $A, B, C, D$ for clarity. We next consider the equation $C/A = D/B$, and through this we find the result $\alpha' = \alpha + 1/2$.\\
\\
Finally, from \eqref{randsrelationshipanticommutator}, \eqref{evengrlm}, and \eqref{svone} we have 
\begin{equation}\label{all the equations}
 a_rc_s + a_sc_r = 2f(r + s), \; f(m) c_r = c_{r+m},\; f(m)a_r = a_{m + r}, 
\end{equation}
for all $m \in \Gamma_0, r \in \Gamma_1$. In particular, we see 
\[
\frac{c_{r+m}}{c_r} = \frac{a_{m + r}}{a_r}. 
\]
As $a_{r_1} = c_{r_1}$ we obtain that $a_r = c_r$ for every $r \in \Gamma_1$. Thus, we may define the function $g: \Gamma \ra \CC$ by
\begin{align*}
g(m) &\defeq f(m), \; m \in \Gamma_0, \\
g(r) &\defeq a_r, \; r \in \Gamma_1
\end{align*}
and from \eqref{all the equations} we conclude $g \in \Gamma^{\lor}$, as required. Therefore, the action of $G_r$ on $P(L_0) + \xi Q(L_0)$ is given by 
\[
G_r(P(L_0) + \xi Q(L_0)) = g(r)(L_0 + 2r\alpha)Q(L_0 + r) + g(r)\xi P(L_0 + r)
\]
and 
$M \cong s\Omega_\Gamma(g, \alpha)$. If instead of \eqref{qrandpr} we take the choice of 
\[
p_r(L_0) = a_r(L_0 + b_r), \; \; q_r(L_0) = c_r, \; \; a_r, c_r \in \CC^*, \; \;  b_r \in \CC
\]
for $r \in \Gamma_1$, we may follow an identical argument to conclude $M \cong \Pi(s\Omega_\Gamma(f, \alpha))$. This concludes the theorem.
\end{proof}

% \begin{corollary}
%     Let $M$ be a $s\mathcal{W}_\Gamma$--module such that $L_0$ acts by multiplication. Then $M = s\Omega_\Gamma(f, \alpha)$ for $\alpha \in \CC$, $f \in \Gamma^{\lor}$. 
% \end{corollary}
% \begin{proof}
%     This follows directly from the theorem. 
% \end{proof}

%%%%%%%%%%%%%%%%%%%%%%%%%%%%%%%%%%%%%%%%%%%%%%%%%%%%%%%%%%%%%%%%%%%%%%%%%%%%%%%%%%%%%%%%%%%%%%%%%%

\subsection{Twisted module}

%%%%%%%%%%%%%%%%%%%%%%%%%%%%%%%%%%%%%%%%%%%%%%%%%%%%%%%%%%%%%%%%%%%%%%%%%%%%%%%%%%%%%%%%%%%%%%%%%%

In this section we consider the modules $\Omega_\Gamma(f, \alpha)$ and $s\Omega_\Gamma(f, \alpha)$ after a twist by an automorphism. 

\begin{proposition}\label{twisted module}
    \begin{enumerate}[i.] 
    
     \item Let $\varphi_{a, f} \in \aut(\mathcal{W}_\Gamma)$, and let $\Omega_\Gamma(g, \alpha)^{a, f}$ denote $\Omega_\Gamma(g, \alpha)$ after a twist by $\varphi_{a, f}$. Then $\Omega_\Gamma(g, \alpha)^{a, f} = \Omega_\Gamma(g', \alpha)$ where $g'(m) \defeq f(m)g(m/a)$. 
     
     \item Let $\varphi_{a, f} \in \aut(s\mathcal{W}_\Gamma)$, and let $s\Omega_\Gamma(g, \alpha)^{a, f}$ denote $s\Omega_\Gamma(g, \alpha)$ after a twist by $\varphi_{a, f}$. Then $s\Omega_\Gamma(g, \alpha)^{a, f} = s\Omega(g', \alpha)$ where $g'(m) \defeq f(m)g(m/a^2)$. 
     
    \end{enumerate}
\end{proposition}
\begin{proof}
    (i) \; We note that $\Omega_\Gamma(g, \alpha)^{a, f}$ must be of the form $\Omega_\Gamma(g', \alpha')$ for some $\alpha' \in \CC$, $g' \in \Gamma^{\lor}$. Next, consider the map
    \[
    \psi: P(L_0) \mapsto P(L_0'/a). 
    \]
    Using the fact that $\varphi_{a, f}(L_m)\psi(P(L_0)) = \psi(\varphi_{a, f}(L_m)P(L_0))$ it is easy to see 
    \begin{align*}
    % \psi(\varphi_{a, f}(L_m)P(L_0)) &= \psi(af(m)L_{m/a}P(L_0)) \\
    % &= f(m)g(m/a)(L_0' + m\alpha)P(\frac{L_0' + m}{a}) \\
    % \implies 
    \varphi_{a,f}(L_m)P\left(\frac{L_0'}{a}\right) = f(m)g(m/a)(L_0' + m\alpha)P\left(\frac{L_0' + m}{a}\right).
    \end{align*}
    Defining $g'(m) \defeq f(m)g(m/a)$ we conclude (i). \\
    \\
    (ii) \; We again have that $s\Omega_\Gamma(g, \alpha)^{a, f}$ must be of the form $s\Omega_\Gamma(g', \alpha')$ for some $\alpha' \in \CC$, $g' \in \Gamma^{\lor}$. Fix $P(L_0) \in (s\Omega_\Gamma)_{\overline{0}}, \; \xi Q(L_0) \in (s\Omega_\Gamma)_{\overline{1}}$, and let us consider the map 
    \[
    \psi: P(L_0) + \xi Q(L_0) \mapsto P(L_0'/a^2) + (\xi'/a)Q(L_0'/a^2).
    \]
    Using an identical method to that in (i), one finds 
    \begin{align*}
    %     \psi(\varphi_{a,f}(G_r)P(L_0)) &= \psi\left(af(r)g(\frac{r}{a^2})\xi P\left(L_0 + \frac{r}{a^2}\right)\right) \\
    %     &= f(r)g(\frac{r}{a^2}) \xi' P\left(\frac{L_0' + r}{a^2}\right) \\
    %   \implies 
       \varphi_{a,f}(G_r)P\left(\frac{L_0'}{a^2}\right) &= f(r)g\left(\frac{r}{a^2}\right) \xi ' P\left(\frac{L_0' + r}{a^2}\right) \\
       \varphi_{a,f}(G_r)\frac{\xi'}{a}Q\left(\frac{L_0'}{a^2}\right) &= \frac{1}{a}f(r)g\left(\frac{r}{a^2}\right)(L_0' + 2r\alpha)Q\left(\frac{L_0' + r}{a^2}\right). 
    \end{align*}
    % and
    % \begin{align*}
    %     \psi(\varphi_{a,f}(G_r)\xi Q(L_0)) &= \psi\left(af(r)g(\frac{r}{a^2})(L_0 + \frac{2r}{a^2}\alpha)Q(L_0 + \frac{r}{a^2})\right) \\
    %     &= \frac{1}{a}f(r)g(\frac{r}{a^2})(L_0' + 2r\alpha)Q\left(\frac{L_0' + r}{a^2}\right) \\ 
    %     \implies 
    % \end{align*}
    Finding the action $\varphi_{a, f}(L_m) \psi(P(L_0) + \xi Q(L_0))$ is done in an identical manner. Defining $g'(m) \defeq f(m)g(m/a^2)$ we conclude (ii). 
\end{proof}

%%%%%%%%%%%%%%%%%%%%%%%%%%%%%%%%%%%%%%%%%%%%%%%%%%%%%%%%%%%%%%%%%%%%%%%%%%%%%%%%%%%%%%%%%%%%

\subsection{Simplicity}\label{simplicity}

%%%%%%%%%%%%%%%%%%%%%%%%%%%%%%%%%%%%%%%%%%%%%%%%%%%%%%%%%%%%%%%%%%%%%%%%%%%%%%%%%%%%%%%%%%%%%%%%%

By \cite{tan2015wn+}, \cite{Yang2019} the $\mathcal{W}_\ZZ$--modules $\Omega_\ZZ(\lambda, \alpha)$ and the $s\mathcal{W}_\ZZ$--modules $s\Omega_\ZZ(\lambda, \alpha)$ are simple for $\alpha \neq 0$. In this section we show analogous results for the $\mathcal{W}_\Gamma$--modules $\Omega_\Gamma(f, \alpha)$ and $s\mathcal{W}_\Gamma$--modules $s\Omega_\Gamma(f, \alpha)$. 
\begin{theorem}
    \begin{enumerate}[i.]
    
        \item $\Omega_\Gamma(f, \alpha)$ is simple if and only if $\alpha \neq 0$. 
        
        \item A module $\Omega_\Gamma(f, 0)$ has a unique simple submodule $N$ with $N \cong \Omega_\Gamma(f, 1)$ and $\Omega_\Gamma(f, 0)/N$ a trivial module. 
        
    \end{enumerate}
    
\end{theorem}
\begin{proof} 
    (i) \; By \Cref{twisted module}, we may assume $f \equiv 1$. Let $M$ be a non-zero submodule of $\Omega_\Gamma(f, \alpha)$. Let $P(L_0) \in M$ be a polynomial of minimal degree, denote this degree by $s$. Then necessarily for each $m \in \Gamma$ the element 
    \[
    (L_m - L_0)P(L_0) = (L_0 + m\alpha)P(L_0 + m) - L_0 P(L_0)
    \]
    is in $M$ and is of degree at most $s$. As $s$ is the minimal degree, we must have that 
    \[
    (L_m - L_0)P(L_0) = c_m P(L_0)
    \]
    for some constant $c_m$. This gives us 
    \[
    (L_0 + m\alpha) P(L_0 + m) = (L_0 + c_m)P(L_0)
    \]
    Let $z \in \CC$ be a root of $P(L_0)$. Then we know 
    \[
    (z + m \alpha) P(z + m) = 0
    \]
    which means that for every $m$ such that $m\alpha \neq -z$, $P(z+ m) = 0$, which implies $P(L_0) = 0$. Therefore, $M$ does not contain any non-constant polynomials and we conclude (i).
    % Let $M$ be a non-zero submodule of $\Omega_\Gamma(f, \alpha)$, and let $P(L_0) = \sum_{i = 0}^s c_iL_0^i \in M$ be a polynomial of minimal degree. Then the element
    % \[
    % (L_m - L_0)P(L_0) = (L_0 + \alpha m) P(L_0 + m) - L_0 P(L_0)
    % \]
    % lives in $M$ and is of degree at most $s$. So for each $m$, the polynomial 
    % \begin{equation*}
    % (L_0 + \alpha m) P(L_0 + m) - L_0 P(L_0) = \sum_{i = 0}^s c_i(L_0 + \alpha m)(L_0 + m)^i - \sum_{i = 1}^s c_iL_0^{i + 1}
    % \end{equation*}
    % is of degree s. Furthermore, we must have that they are proportional. Looking at the degree $s$ term we find 
    % $$
    % mc_s(s + b)L_0^s 
    % $$
    % which tells us
    % $$
    % (L_m - L_0)P(L_0) = m(L_1 - L_0)P(L_0). 
    % $$
    % Now we consider the constant term, 
    % \begin{align*}
    % c_0m\alpha + c_1m^2 \alpha + c_2 m^3\alpha + \ldots + c_s m^{s+1} \alpha &= m\alpha \sum_{i=0}^s c_i m^i = m\alpha P(1)
    % \end{align*}
    % since we found that the constant of proportionality is $m$. Then we see 
    % $
    % P(m) = P(1)
    % $
    % for all $m$, which implies 
    % $
    % P(L_0) = K
    % $
    % for some constant $K$. Therefore, $M$ contains everything in $\Omega_\Gamma(f,\alpha)$, and $\Omega_\Gamma(f, \alpha)$ is simple. 
    \\
   \\
   (ii) \; Let $N \subset \Omega_\Gamma(f, 0)$ be the set consisting of polynomials with zero constant term. It is an easy check that $N$ is a submodule. 
    It remains to show that $N \cong \Omega_\Gamma(f, 1)$, which implies simplicity. Let $1,$ $1'$ be the generators of free $\CC[L_0]$--modules 
    \[
    \Omega_\Gamma(f, 1), \; \; \Omega_\Gamma(f, 0)
    \]
    respectively. Consider the map $\varphi: \Omega_\Gamma(f, 1) \ra \Omega_\Gamma(f, 0)$ defined by 
    \[
    \varphi(P(L_0)1) = P(L_0) L_0 1'.
    \]
    Then 
    \[
    \varphi(L_mP(L_0)1) = \varphi(f(m)(L_0 + m\alpha)P(L_0 + m)1) = f(m)(L_0 + m\alpha)P(L_0) L_01' = L_m \varphi(P(L_0)1)
    \]
    which tells us $\varphi$ is a monomorphism. Therefore, $\Omega_\Gamma(f, 1) \cong N$. 
\end{proof}

% \begin{proposition}
%     $\Omega_\Gamma(f, 0)$ is not simple, and has a unique simple submodule $\Omega_\Gamma(f, 1)$. 
% \end{proposition}

\begin{theorem}
\begin{enumerate}[i.]
    
    \item $s\Omega_\Gamma(f, \alpha)$ is simple if and only if $\alpha \neq 0$.
    
    \item A module $s\Omega_\Gamma(f, 0)$ has a unique simple submodule $N$ with $N \cong \Pi(s\Omega_\Gamma(f, 1/2))$ and $s\Omega_\Gamma(f, 0)/N$ a trivial module. 

\end{enumerate}
\end{theorem}
\begin{proof}
    (i) \; Assume that $N = N_{\overline{0}} \oplus N_{\overline{1}}$ is a non-zero submodule of $s\Omega_\Gamma(f, \alpha)$, and let us denote $f_0 \defeq f\big|_{\Gamma_0}$. Then we must have that 
    \begin{equation*}\label{subone}
    N_{\overline{0}} \subseteq s\Omega_\Gamma(f, \alpha)_{\overline{0}}
    \end{equation*}
    is a submodule of $\Omega_{\Gamma_0}(f_0, \alpha)$ and 
    \begin{equation*}\label{subtwo}
    N_{\overline{1}} \subseteq s\Omega_\Gamma(f, \alpha)_{\overline{1}}
    \end{equation*}
    is a submodule of $\Omega_{\Gamma_0}(f_0, \alpha + 1/2)$. However, for $\alpha \neq 0, -1/2$ we know that both $\Omega_{\Gamma_0}(f_0, \alpha)$ and $\Omega_{\Gamma_0}(f_0, \alpha + 1/2)$ are simple. 
    % This tells us that \eqref{subone} is 
    % $
    % \Omega_\Gamma(f_0, \alpha) 
    % $
    % and \eqref{subtwo} is 
    % $
    % \Omega_\Gamma(f_0, \alpha). 
    % $
    From this, it is easy to see that either $N$ is either a trivial submodule or $N = s\Omega_\Gamma(f, \alpha)$.  \\
    \\
    Assume now that $\alpha = -1/2$, then $s\Omega_\Gamma(f, -1/2)_{\overline{1}} = \Omega_{\Gamma_0}(f_0, 0)$ which we know is not simple. Let $A$ be the non-zero proper submodule of $\Omega_{\Gamma_0}(f_0,0)$. Then there are two additional possibilities for submodules, $N =  s\Omega_\Gamma(f, -1/2)_{\overline{0}} \oplus A$ or $N = N_{\overline{1}} = A$. It is easy to check that neither are submodules and therefore, $s\Omega_\Gamma(f, \alpha)$ is simple for $\alpha \neq 0$. \\
    \\
    (ii) \; Assume $\alpha = 0$ and let us again define $f_0 \defeq f\big|_{\Gamma_0}$ We note that 
    \[
        s\Omega_\Gamma(f, 0)_{\overline{0}} = \Omega_{\Gamma_0}(f_0, 0), \; s\Omega_\Gamma(f, 0)_{\overline{1}} = \Omega_{\Gamma_0}(f_0, 1/2). 
    \]
    We have shown that $\Omega_{\Gamma_0}(f_0, 0)$ is not simple; we have a unique submodule $A \cong \Omega_{\Gamma_0}(f_0, 1)$. It is a simple check to see that $N = N_{\overline{0}} \oplus N_{\overline{1}}$ with 
    \[
        N_{\overline{0}} \cong \Omega_{\Gamma_0}(f_0, 1), \; \;  N_{\overline{1}} \cong \Omega_{\Gamma_0}(f_0, 1/2) 
    \]
    is a submodule of $s\Omega_{\Gamma}(f, 0)$, and from here it is straightforward to see that the quotient by $N$ will be a trivial module. \\
    \\
    Moreover, let us show that $N \cong \Pi(s\Omega_\Gamma(f, 1/2))$. 
    Define $\varphi: \Pi(s\Omega_\Gamma(f, 1/2)) \ra s\Omega_\Gamma(f, 0)$ by 
    \[
    \varphi(P(L_0)1) \defeq P(L_0) L_0 1', \; \; \varphi(P(L_0) \xi) \defeq P(L_0) \xi'
    \]
    where $1$, $\xi$ are free generators of $\Pi(s\Omega(f, 1/2))$ viewed as a $\CC[L_0] \oplus \xi \CC[L_0]$ module, and similarly for $1'$, $\xi'$ and $s\Omega_\Gamma(f, 0)$. Then we see that 
    \begin{align*}
    \varphi(G_r P(L_0) 1) &= \varphi(f(r)(L_0 + r) P(L_0 + r) \xi) = f(r)(L_0 + r)P(L_0 + r)\xi' \\
    G_r\varphi(P(L_0) 1) &= G_rP(L_0)L_0 1' = f(r)P(L_0 + r)(L_0 + 1) \xi' \\
    \varphi(G_r P(L_0) \xi) &= \varphi(f(r) P(L_0 + r) 1) = f(r)P(L_0 + r) L_0 1' \\
    G_r \varphi(P(L_0) \xi) &= G_r P(L_0) \xi' = f(r) L_0 P(L_0 + r) 1'.
    \end{align*}
    An identical process will confirm that $L_m\varphi(P(L_0) 1) = \varphi(L_m P(L_0)1)$ and $L_m \varphi(P(L_0) \xi) = \varphi(L_m P(L_0) \xi)$. Therefore, $\varphi$ is a monomorphism, and we conclude $N \cong \Pi(s\Omega_\Gamma(f, 1/2)$. 
    
    \end{proof}

\section{Weighting Functor and Intermediate Series}\label{weighting functor section}

%%%%%%%%%%%%%%%%%%%%%%%%%%%%%%%%%%%%%%%%%%%%%%%%%%%%%%%%%%%%%%%%%%%%%%%%%%%%%%%%%%%%%%%%%%%%%

In this section we apply the weighting functor to the modules $\Omega_\Gamma(f, \alpha)$ and $s\Omega_\Gamma(f, \alpha)$ and show that the resulting modules are of the intermediate series. We begin by defining the weighting functor, first introduced by J. Nilsson in 2016. 

\begin{definition}
Let $\fg$ be a Lie (super)algebra and let $\fh \subset \fg$ be a commutative subalgebra such that $\Ad(\fh)$ acts diagonally on $\fg$, in other words
\[
\fg \defeq \bigoplus_{\mu \in \fh^*} \fg_{\mu}.
\]
View each $\mu \in \fh^*$ as a homomorphism $\overline{\mu}: \mathcal{U}(\fh) \ra \CC$. For a $\fg$--module $M$ consider the vector space 
\begin{equation}\label{weightingfunctor}
    W(M) \defeq \bigoplus_{\fm \in \text{Max}(\mathcal{U}(\fh))} \faktor{M}{\fm M} = \bigoplus_{\mu \in \fh^*} \faktor{M}{\ker(\overline{\mu}) M}
\end{equation}
Take an $\alpha \in \fh^*$ and $v \in M$. We define an action of $x_{\alpha} \in \fg_{\alpha}$ on $W(M)$ by 
\begin{equation}\label{weightingaction}
    x_\alpha (v + \ker(\overline{\mu}) M) \defeq x_\alpha v + \ker (\overline{\mu + \alpha}) M.
\end{equation}
We call the map $W$ between $\mathcal{U}(\fg)$-modules the \textit{weighting functor}. By \cite{Nilsson2016}, the assignment $M \mapsto W(M)$ is a covariant functor $W: \mathcal{U}(\fg) \ra \mathcal{U}(\fg)$. That is, for two $\fg$--modules $M, N$ each homomorphism $$\psi: M \ra N$$ induces a homomorphism $W(\psi): W(M) \ra W(N)$ by the formula
\[
W(\psi)(v + \ker(\overline{\mu})M) \defeq \psi(v) + \ker(\overline{\mu}) N.
\]
This functor is right exact; in our final remark we show that the functor is not left exact. The purpose of the weighting functor is to assign a weight module to a $\mathcal{U}(\fg)$ module. 
\end{definition}
One can refer to \cite{Nilsson2016} for the full proof that $W$ is indeed a functor. We define the weighting functor $W$ for $\mathcal{W}_\Gamma$ and $s\mathcal{W}_\Gamma$ by taking $\fh \defeq \CC L_0$. 

We now introduce the intermediate series modules and super intermediate series modules. For $\mathcal{W}_\ZZ$ and $s\mathcal{W}_\ZZ, s\mathcal{W}_{\ZZ + 1/2}$ these modules were first introduced by V. G. Kac and Y. Su, respectively. For these definitions see \cite{Mathieu1992, Yucai1995}. 
 
\begin{definition}[Intermediate series module]\label{expandedintermediatemodule}
Fix $a \in \CC$. An \textit{intermediate series module} $V_\CC(a)$ over $\mathcal{W}_\Gamma$ is a module with a basis $\{v_k\}_{k \in \CC}$ and the action
\begin{equation}\label{intermediateaction}
L_m v_k \defeq (k - am)v_{k - m}, \; \; m \in \Gamma.
\end{equation}
\end{definition}

% \begin{proposition}\label{vcb=vcb'}
%     $V(a,b) \cong V(a, b')$ for $b \neq b'$. 
% \end{proposition}
% \begin{proof}
%     Under the map $w_k \mapsto v_{k + c}$ for $c \in \CC$, we have $V(a,b) \cong V(a, c+b)$ for $c \in \CC$. Thus $V(a,b) \cong V(a, b')$. 
% \end{proof}

% \begin{note}
% Since we lose dependence on $b$ for the extended intermediate series modules, let us adopt the notation 
% \[
% V(a,b) \defeq V(a) . 
% \]
% \end{note}

\begin{definition}[Super intermediate series module]
    Fix $a \in \CC$. A \textit{super intermediate series module} $sV_\CC(a)$ over $s\mathcal{W}_\Gamma$ is a module with a basis $\{v_k\}_{k \in \CC}$ in $sV_\CC(a)_{\overline{0}}$, $\{w_k\}_{k \in \CC}$ in $sV_\CC(a)_{\overline{1}}$ and the actions
    \begin{align}
        L_mv_k &\defeq (k - am)v_{k - m}, \; \;\; \; \; L_mw_j \defeq (j - (a - 1/2)m)w_{j - m} \label{evensuperintmedseriesaction} \\
        G_rv_k &\defeq w_{k - r}, \; \; \; \; \; \; \; \; \; \; \; \; \; \; \; \; \; \; \; \; G_rw_j \defeq (j - 2r(a - 1/2))v_{j - r} \label{oddsuperintermediateseriesmodaction}
    \end{align}
    for all $m \in \Gamma_0$, $r \in \Gamma_1$. We note that 
    \[
    sV_\CC(a)_{\overline{0}} \cong V_\CC(a), \; \; \; \;sV_\CC(a)_{\overline{1}} \cong V_\CC(a - 1/2).
    \]
\end{definition}

\begin{remark} 
As a $\mathcal{W}_\Gamma$ module, $V_\CC(a)$ is not simple for $\Gamma \neq \CC$. As a $\mathcal{W}_\ZZ$--module one has 
\[
V_\CC(a) = \bigoplus_{b \in \CC/\ZZ} V_\ZZ(a, b)
\]
where $V_\ZZ(a,b)$ is the intermediate series module introduced by V. G. Kac. One may refer to \cite{Mathieu1992} for the definition of $V_\ZZ(a,b)$. 
Similarly, as an $s\mathcal{W}_\ZZ$ module, 
\[
sV_\CC(a) = \bigoplus_{b \in \CC/\ZZ} sV_\ZZ(a, b) 
\]
and as an $s\mathcal{W}_{\ZZ + 1/2}$ module, 
\[
sV_\CC(a) = \bigoplus_{b \in \CC/\ZZ} sV_{\ZZ + 1/2}(a, b)
\]
where $sV_\ZZ$ and $sV_{\ZZ + 1/2}$ are the Ramond and Neveu-Schwarz algebras, respectively. One may refer to \cite{Yucai1995} for the full definition of $sV_\ZZ(a, b)$ and $sV_{\ZZ + 1/2}(a,b)$. 
\end{remark}

%%%%%%%%%%%%%%%%%%%%%%%%%%%%%%%%%%%%%%%%%%%%%%%%%%%%%%%%%%%%%%%%%%%%%%%%%%%%%%%%%%%%%%%%%%%%%

\subsection{Theorem}

%%%%%%%%%%%%%%%%%%%%%%%%%%%%%%%%%%%%%%%%%%%%%%%%%%%%%%%%%%%%%%%%%%%%%%%%%%%%%%%%%%%%%%%%%%%%%

\begin{theorem}\label{in this we show lambda can be 1} 
\begin{enumerate}[i.]
    
    \item If $\tilde{f} \in \Gamma^{\lor}$ extends to $f \in \CC^{\lor}$ (i.e., $f\big|_{\Gamma} = \tilde{f}$), then 
    \[
     W(\Omega_\Gamma(\tilde{f}, \alpha)) \cong V_\CC(1-\alpha).
    \]
    
    \item  A $\mathcal{W}_\CC$ module $V_\CC(\alpha)$ is simple if and only if $\alpha \neq 0, 1$. The $\mathcal{W}_\CC$ module $V_\CC(0)$ has a unique proper submodule $\CC v_0$, and the module $V_\CC(1)$ has a unique proper submodule $\oplus_{z \in \CC^*} \CC v_z.$

\end{enumerate}
\end{theorem}
\begin{proof}
    (i) \;Let $\Omega_\Gamma(f, \alpha) = M$ be a $\mathcal{W}_\Gamma$ module. For each $z \in \CC$ let $\mu_z \in \fh^*$ denote the homomorphism defined by $L_0 \mapsto z$. Then we know $\ker(\overline{\mu_z}) = \CC[L_0](L_0 - z).$ Define 
    \[
    v_c \defeq 1 + \ker(\overline{\mu_c})M,
    \]
    and in particular we note that it spans the weight space $W(\Omega_\Gamma(\tilde{f}, \alpha))_c$. The vectors $v_c$ for $c \in \CC$ form a basis of $W(\Omega_\Gamma(\tilde{f}, \alpha))$. One has 
    \[
    L_m v_c = \tilde{f}(m)(L_0 + m\alpha)1 + \ker(\overline{\mu_{c-m}}) M = \tilde{f}(m)(c - m(1 - \alpha))v_{c - m}
    \]
    and by setting $y_c \defeq f(-c) v_c$ we obtain
    \[
    L_m y_c = f(-(c - m))(c- m(1-\alpha))v_{c - m} = (c - m(1-\alpha))y_{c - m}.
    \]
    Therefore, we conclude $W(\Omega_\Gamma(f, \alpha)) \cong V_\CC(1 - \alpha)$, as desired. \\
    \\
    (ii) \;Now let us show that $V_\CC(\alpha)$ is simple for all $\alpha \neq 0, 1$. In the case of $\alpha = 1$, consider the set $\oplus_{z \in \CC^*} \CC v_z$. We see that $L_m v_k = (k-m)v_{k-m}$ and we cannot arrive at $v_0$, confirming this is a submodule. Furthermore, from $v_0$ we generate the whole module, so there are no other submodules. 
    
    For $\alpha = 0$ we have the submodule $\CC v_0$ as 
    $
    L_m v_0 = 0
    $
    for any $m \in \CC$. Moreover, from any other element $v_k$ we can arrive at $v_0$, confirming that there are no other submodules. Lastly, it is easy to see that actions on the element $v_0$ generate all of $V_\CC(\alpha)$ if $\alpha \neq 0, 1$. Similarly, $L_k v_k$ will bring us to $v_0$, and we conclude that $V_\CC(\alpha)$ is simple if and only if $\alpha \neq 0, 1$. 
\end{proof}

% \begin{corollary}
% $W(\Omega(1, 0)) \cong V_\CC(0)$ is not simple. 
% \end{corollary}
% \begin{proof}
%     Consider the set $\langle a_k \rangle$ for all $k \in \CC^*$. Then we see that $L_m a_k = (k-m)a_{k-m}$ and we cannot arrive at $e_0$, confirming this is a submodule. Furthermore, from $a_0$ we generate the whole group, so there are no other submodules. 
% \end{proof}

% \begin{corollary}
%     $W(\Omega(1,1)) \cong V_\CC(1)$ is not simple.
% \end{corollary}
% \begin{proof}
%     We have the submodule $\langle a_0 \rangle$ as 
%     $
%     L_m a_0 = 0
%     $
%     for any $m \in \CC$. Moreover, from any other element $a_k$ we can arrive at $a_0$, confirming that there are no other submodules. 
% \end{proof}

\begin{theorem}
    \begin{enumerate}[i.]
    
    \item If $\tilde{f} \in \Gamma^{\lor}$ extends to $f \in (\CC \times \ZZ_2)^{\lor}$, then 
    \[
     W(s\Omega_\Gamma(\tilde{f}, \alpha)) \cong sV_\CC(1-\alpha).
    \]
    
    \item An $s\mathcal{W}_\CC$ module $sV_\CC(\alpha)$ is simple if and only if $\alpha \neq 1/2, 1$. The $s\mathcal{W}_\CC$ module $sV_\CC(1/2)$ has a unique proper submodule $\CC w_0$, and the module $sV_\CC(1)$ has a unique proper submodule $\oplus_{k \in \CC^*} \CC v_k$. 
    
    \end{enumerate}
\end{theorem}
\begin{proof}
    (i) Let $s\Omega_\Gamma(\tilde{f}, \alpha) = M$ be an $s\mathcal{W}_\Gamma$ module. For each $z \in \CC$ let $\mu_z \in \fh^*$ denote the homomorphism defined by $L_0 \mapsto z$. Then $\ker(\overline{\mu_z}) = \CC[L_0](L_0 - z)$. Define 
    \[
    v_c \defeq 1 + \ker(\overline{\mu_c})M, \; \; w_c \defeq \xi + \ker(\overline{\mu_c})M.
    \]
    The vectors $v_c, w_c$ for $c \in \CC$ form a basis of $W(s\Omega_\Gamma(\tilde{f}, \alpha))$. We find the actions by $s\mathcal{W}_\Gamma$ to be
    \begin{align*}
    L_mv_c &= \tilde{f}(m;0)(c - m(1 - \alpha)) v_{c - m}, \; \; \; \; \; L_m w_c = \tilde{f}(m;0)(c - m(1/2 - \alpha)w_{c - m} \\ 
    G_rv_c &= \tilde{f}(r;1)w_{c - r}, \; \; \; \; \; \; \; \; \; \; \; \; \; \; \; \; \; \; \; \; \; \; \; \; \; \; \; \; \; \; G_r w_c = \tilde{f}(r;1)(c - 2r(1/2 - \alpha))v_{c - r}. 
    \end{align*}
    Setting $\tilde{v}_c = f(-c;0)v_c$, $\tilde{w}_c = f(-c;1)w_c$ we find 
    \begin{align*}
    G_r \tilde{v}_c &= f(-c;0) G_r1 + \ker(\overline{\mu_{c -r}})M = f(-c;0) \tilde{f}(r;1) w_{c-r} = \tilde{w}_{c - r},\\
    G_r\tilde{w}_c &= f(-c; 1) G_r \xi + \ker(\overline{\mu_{c - r}})M = f(-c;1) \tilde{f}(r;1)(c - r(1 - 2\alpha)) v_{c - r} = (c - 2r(1 - \alpha)) \tilde{v}_{c - r}.
    \end{align*}
    Therefore, we conclude $W(s\Omega_\Gamma(f, \alpha)) \cong sV_\CC(1-\alpha)$. \\ 
    \\
    (ii) \;Considering simplicity, let us recall that $sV_\CC(\alpha)_{\overline{0}} \cong V_\CC(\alpha), sV_\CC(\alpha)_{\overline{1}} \cong V_\CC(\alpha - 1/2)$. Moreover, we know that $V_\CC(\alpha)$ is simple if and only if $\alpha \neq 0, 1$. If $N$ is a non-zero submodule of $sV_\CC(\alpha)$, then
    \[
    N_{\overline{0}} \subset sV_\CC(\alpha)_{\overline{0}}, \; \; N_{\overline{1}} \subset sV_\CC(\alpha)_{\overline{1}}
    \]
    are necessarily submodules of $V_\CC(\alpha), V_\CC(\alpha - 1/2)$ respectively. Then we just need to check the cases of $\alpha = 0, 1, 1/2, 3/2$. 
    
    If $\alpha = 0$, then $sV_\CC(0)_{\overline{0}} \cong V_\CC(0)$ has a trivial submodule spanned by $v_0$. This gives us two options for $N$: 
    \[
    N = N_{\overline{0}} = \CC v_0 \; \text{ and } \; N_{\overline{0}} = \CC v_0, \; \;  N_{\overline{1}} = sV_\CC(0)_{\overline{1}}.
    \]
    It is a simple check to see that neither are submodules under action by $G_r$. 
    
    Next, let us consider $\alpha = 3/2$. Then $sV_\CC(3/2)_{\overline{1}} \cong V_{\CC}(1)$ contains the submodule $\oplus_{k \in \CC^*} \CC w_k$. So our choices for a submodule $N$ are 
    \[
   N = N_{\overline{1}} = \bigoplus_{k \in \CC^*} \CC w_k \;  \text{ and } \; N_{\overline{0}} = sV_\CC(3/2)_{\overline{0}}, \; \; N_{\overline{1}} = \bigoplus_{k \in \CC^*} \CC W_k.
    \]
    It is again easy to see that neither are submodules under action by $G_r$ \\
    \\
    Lastly, both $sV_\CC(1)$ and $sV_\CC(1/2)$ are not simple. For $sV_\CC(1/2)$ we consider the subset $N$ for $N_{\overline{0}} = 0, N_{\overline{1}} = \CC w_0$. A quick check confirms that 
    \[
    L_m w_0 = 0 \text{ and } G_rw_0 = 0. 
    \]
    As $m \in \CC$, we know that from any other element $w_k$ we can arrive at $w_0$, confirming there are no other submodules. 
    
    Next, in $sV_\CC(1)$ we consider the subset $N$ given by $N_{\overline{0}} = \oplus_{k \in \CC^*} \CC v_k$ and $N_{\overline{1}} = sV_{\CC}(1)_{\overline{1}}$. Notice that 
    \[
    L_mv_k = (k - m) v_{k-m}, \; \; L_m w_j =  (j - \frac{1}{2}m) w_{j - m} 
    \]
    and 
    \[
    G_rv_k = w_{k - r}, \; \; G_rw_j = (j - r)v_{j - r} 
    \]
    for $k \in \CC^*, j \in \CC$. So clearly $N$ is a submodule. Moreover, from $v_0$ we can arrive at $v_k$ for any $k \in \CC$, so $N$ is unique and we conclude the theorem.  
\end{proof}

\begin{remark}\label{last remark} 
By results in \Cref{simplicity}, we have the following short exact sequences 
\begin{align} 
    0 &\ra \Omega_\Gamma(f, 1) \xrightarrow{\psi_1} \Omega_\Gamma(f, 0) \ra \CC \ra 0 \label{omega short exact sequence 1}\\
    0 &\ra \Pi(s\Omega_\Gamma(f, 1/2)) \xrightarrow{{\psi_2}} s\Omega_\Gamma(f, 0) \ra \CC \ra 0 \label{omega short exact sequence 2}
\end{align}
where $\psi_1$, $\psi_2$ are the embeddings defined by 
\[
P(L_0) \mapsto L_0 P(L_0)
\]
and 
\[
P(L_0) + \xi Q(L_0) \mapsto L_0 P(L_0) + \xi'Q(L_0)
\]
respectively. If $f \in \Gamma^{\lor}$ extends to $\tilde{f} \in (\CC \times \ZZ_2)^{\lor}$, applying the weighting functor to \eqref{omega short exact sequence 1} and \eqref{omega short exact sequence 2} induces the exact sequences 
\begin{align*}
    0 &\ra \CC \ra V_\CC(0) \xrightarrow{W(\psi_1)} V_\CC(1) \ra \CC \ra 0 \\
    0 &\ra \CC \ra \Pi(sV_\CC(1/2)) \xrightarrow{W(\psi_2)} sV_\CC(1) \ra \CC \ra 0,
\end{align*}
where $W(\psi_1)$, $W(\psi_2)$ are the maps defined by 
\[
v_c \mapsto c v'_c
\]
and 
\[
v_c \mapsto cv'_c, \; \; \; w_c \mapsto w'_c,
\] 
respectively. This shows that the weighting functor is not left exact.
\end{remark}

% \begin{corollary}
%     $W(s\Omega(1,1/2)) \cong sV_\CC(1/2)$ is not simple.
% \end{corollary}
% \begin{proof}
%     We claim
%     $
%     M = \{0\} \oplus b_0 
%     $
%     is a trivial submodule. A quick check confirms that 
%     \[
%     L_m(0 + b_0) = 0 + 0 \text{ and } G_r(0 + b_0) = 0 + 0. 
%     \]
%     We also know that aside from this, as $m \in \CC$ for $L_m$, we have no other submodules. 
% \end{proof}

% \begin{corollary}
%     $W(s\Omega(1,0)) \cong sV_\CC(0)$ is not simple.
% \end{corollary}
% \begin{proof}
%     Here we have the submodule $M$ where $M_{\overline{0}} = \spn\{a_k\}_{k \in \CC^*}$ and $M_{\overline{1}} = sV_{\CC}(0)_{\overline{1}}$. Notice that 
%     \[
%     L_m(a_k + b_j) = (k - m) a_{k-m} + (j - \frac{1}{2}m)b_{j - m} 
%     \]
%     and 
%     \[
%     G_r(a_k + b_j) = (j - r)b_{j - r} + a_{k - r}
%     \]
%     for $k \in \CC^*, j \in \CC$. 
% \end{proof}

% \begin{corollary}
%     The $s\mathcal{W}_\CC$--modules $W(s\Omega(e^\lambda,1/2))$ and $W(s\Omega(e^\lambda,1))$ are not simple.
% \end{corollary}
% \begin{proof}
%     It is easy to see that 
%     $
%     \{0\} \oplus e_0 
%     $
%     is a trivial submodule, confirming that $W(N(e^\lambda,1/2))$ is not simple. We also know that aside from this, as $m \in \CC$ for $L_m$, we have no other submodules. 
% \end{proof}

% \printbibliography

\bibliographystyle{amsalpha}
\bibliography{biblio.bib}

\providecommand{\bysame}{\leavevmode\hbox to3em{\hrulefill}\thinspace}
\providecommand{\MR}{\relax\ifhmode\unskip\space\fi MR }
\providecommand{\MRhref}[2]{%
  \href{http://www.ams.org/mathscinet-getitem?mr=#1}{#2}
}
\providecommand{\href}[2]{#2}
\begin{thebibliography}{YYX19}

\bibitem[CY18]{chen2018non}
Qiu-Fan Chen and Yu-Feng Yao, \emph{Non-weight modules over algebras related to
  the virasoro algebra}, Journal of Geometry and Physics \textbf{134} (2018),
  11--18.

\bibitem[Kac82]{kac1982some}
Victor~G. Kac, \emph{Some problems on infinite dimensional lie algebras and
  their representations}, Lie algebras and related topics, Springer, 1982,
  pp.~117--126.

\bibitem[KK79]{kac1979structure}
Victor~G. Kac and David~A. Kazhdan, \emph{Structure of representations with
  highest weight of infinite-dimensional lie algebras}, Advances in Mathematics
  \textbf{34} (1979), no.~1, 97--108.

\bibitem[LZ14]{lu2014irreducible}
Rencai Lu and Kaiming Zhao, \emph{Irreducible virasoro modules from irreducible
  weyl modules}, Journal of Algebra \textbf{414} (2014), 271--287.

\bibitem[Mat92]{Mathieu1992}
Olivier Mathieu, \emph{Classification of harish-chandra modules over the
  virasoro lie algebra}, Inventiones mathematicae \textbf{107} (1992), no.~1,
  225--234.

\bibitem[Maz00]{mazorchuk1997classification}
Volodymyr Mazorchuk, \emph{Classification of simple harish-chandra modules over
  {$\mathbb{Q}$}-virasoro algebra}, Math. Nachr. \textbf{209} (2000), 171--177.

\bibitem[Nil16]{Nilsson2016}
Jonathan Nilsson, \emph{{$\mathcal{U}(\mathfrak{h})$}-free modules and coherent
  families}, Journal of Pure and Applied Algebra \textbf{220} (2016), no.~4,
  1475 -- 1488.

\bibitem[Su95]{Yucai1995}
Yucai Su, \emph{Classification of harish-chandra modules over the
  super-virasoro algebras}, Communications in Algebra \textbf{23} (1995),
  no.~10, 3653--3675.

\bibitem[TZ15]{tan2015wn+}
Haijun Tan and Kaiming Zhao, \emph{{$\mathcal{W}_n^+$}-and
  {$\mathcal{W}_n$}-module structures on {$\mathcal{U}(\mathfrak{h}_n)$}},
  Journal of Algebra \textbf{424} (2015), 357--375.

\bibitem[YYX19]{Yang2019}
Hengyun Yang, Yufeng Yao, and Limeng Xia, \emph{A family of non-weight modules
  over the super-virasoro algebras}, 2019.

\end{thebibliography}

\end{document}